\documentclass[preprint,12pt]{elsarticle}

\usepackage{lineno,hyperref}
\usepackage{amssymb,amsmath,amsthm}
\usepackage{booktabs}  
\usepackage{multirow}
\usepackage{rotating}
\usepackage{epsf,epstopdf}
\usepackage{graphicx,graphics}
\usepackage{color}
\usepackage[skip=2pt,font = scriptsize]{caption}
\usepackage{float}
\usepackage{adjustbox}
\numberwithin{equation}{section}
\usepackage{algorithm}
\usepackage{graphics,epstopdf,graphicx,epsf}
\usepackage{algpseudocode}
\usepackage{pifont}
\usepackage{mdframed} 
\usepackage{xparse} 
\newcounter{examplecounter}
\newmdenv[ 
  topline=false,
  bottomline=false,
  rightline=false,
  leftline = false,
  skipabove=\topsep,
  skipbelow=\topsep
]{leftrule}
\NewDocumentEnvironment{example}{O{\textbf{Example~\arabic{examplecounter}\refstepcounter{examplecounter}:~}}} 
{\begin{leftrule}\noindent\textcolor{black}{#1}} {\end{leftrule}}
\newtheorem{thm}{Theorem}
\newtheorem{defn}{Definition}

\newtheorem{lem}{Lemma}

\newtheorem{eg}{Example}
\floatstyle{ruled}
\newfloat{algorithm}{htbp}{loa}
\floatname{algorithm}{Algorithm}
\def\ni{\noindent}

\usepackage[T1]{fontenc}
\usepackage{amsmath}
\DeclareMathOperator*{\argmin}{arg\,min}
\newcommand*{\textcal}[1]{%
  \textit{\fontfamily{qzc}\selectfont#1}%
\modulolinenumbers[5]
}

\journal{TBA}

\begin{document}

\begin{frontmatter}

\title{On Harmonic Ritz vectors and the stagnation of  GMRES}

\author{Mashetti Ravibabu\fnref{myfootnote}}
\address{Indian Institute of Science, Bengaluru, India}


\cortext[mycorrespondingauthor]{Mashetti Ravibabu}
\ead{mashettiravibabu2@gmail.com}


\begin{abstract}
This paper derives a necessary and sufficient condition for the coincidence of Harmonic residual vectors and the residual vector in GMRES. 
The properties of the harmonic Ritz values at the stagnation of GMRES were described in the Proposition-4.2 of \cite{gen}. 
Necessary and sufficient conditions basing on Harmonic Ritz vectors for the stagnation have derived in this paper.
\end{abstract}


\end{frontmatter}
\section{Introduction}\label{intro}
The GMRES method is widely used for approximating the solution of 
sparse  nonsymmetric linear system of equations; see \cite{saad}. GMRES arbitrarily chooses an initial residual vector $r_0,$ and then at each iteration updates a residual vector using the Krylov subspace that based on $r_0$ and the coefficient matrix $A.$ Thus a residual vector can be written as the polynomial in $A,$ so-called the \emph{residual polynomial} acting on $r_0.$ 

The zeroes of the residual polynomial are the Harmonic Ritz values \cite{goo}. Harmonic Ritz values are approximations to eigenvalues of $A$ from the  Krylov subspace in GMRES. The corresponding residual vectors, so-called \emph{Harmonic Residual vectors} are orthogonal to the $A$ image of the Krylov subspace in GMRES. 


If a residual vector of the linear system remains the same for a few consecutive iterations of GMRES then it is called the stagnation. The relations between Harmonic Ritz values in any two successive iterations during the  stagnation
are given in the Proposition-4.2 of \cite{gen}.
This paper shows that the coincidence of GMRES residual vector and harmonic residual vectors is theoretically possible. Then, it derives a necessary and sufficient condition for the stagnation of GMRES based on the Harmonic Ritz vectors. It also discusses the preserve of harmonic Ritz vectors during the stagnation phase of GMRES.

The paper is organized as follows: Section 2 introduces the GMRES method.
In Section 3 a least squares problem is devised that connects solutions of least squares problems in any two successive iterations of GMRES. Using the results of Section-3, Section 4 derives a necessary and sufficient condition for the coincidence of the residual and Harmonic residual vectors. Then Section 5  derives a necessary and sufficient condition for the stagnation of GMRES. Section 6 concludes the paper.

\section{GMRES}\label{gmres}
Consider the following system of linear equations:
$$Ax =b, A \in \textcal{C}^{~n \times n},~ b\in\textcal{C}^{~n}, x \in \textcal{C}^{~n}.$$
Let $x_0 \in \textcal{C}^{~n}$ be an arbitrarily chosen initial
approximation to the solution of the above problem. Without loss of
generality, we assumed throughout the paper that $x_0=0$ so that the initial residual vector is $r_0 =b.$ Then, at the $i^{th}$ iteration of GMRES an approximate
solution belongs to the Krylov subspace:
$$\textcal{K}_i(A,b) ~=~span\{b,Ab,\cdots, A^{i-1} b\}$$
and is of minimal residual norm:
\begin{equation}\label{eq1}\relax
\|r_i\| = \min_{x \in \textcal{K}_i(A,b)} \|b-Ax\|.
\end{equation}
GMRES solves this minimization problem using the following Arnoldi
recurrence relation:
\begin{equation}\label{eq1a}\relax
AV_i=V_iH_i+h_{i+1,i}v_{i+1}e_i^\ast, ~~\mbox{where}~~v_1 =
\frac{b}{\|b\|},
\end{equation}  
where the matrix $V_i=\begin{bmatrix} v_1~ v_2~ \cdots ~v_i
\end{bmatrix},$ and $H_i$ is an unreduced upper Hessenberg matrix of order
$i.$ The vectors $\{v_1,v_2,\cdots,v_j\}$ form an orthonormal basis for $\textcal{K}_j(A,b),$ for $j=1,2,\ldots, i.$ A vector $v_{i+1}$ is of unit norm, and is orthogonal to $v_j$ for $j \leq i$

GMRES uses the equation (\ref{eq1a}) to recast the least squares problem in (\ref{eq1}) into the following:
\begin{equation}\label{eq1b}\relax
z_i = \argmin_{x \in \textcal{C}^i} \|b-AV_ix\|= \arg\min_{x \in \mathcal{C}^i} \|\beta V_{i+1}e_1-V_{i+1}\tilde{H}_i x\|,
\end{equation}
where $\beta = \|b\|,$ and $\tilde{H_i}$ is an upper Hessenberg matrix 
obtained by appending the row $[0~0~\cdots~h_{i+1,i}]$ at the bottom of the matrix $H_i.$ As columns of the matrix $V_{i+1}$ are orthonormal, the above least squares problem is equivalent to the following problem:
$$z_i =\arg\min_{x \in \mathcal{C}^i} \|\beta e_1-\tilde{H_i} x\|.$$
GMRES solves this problem for the vector $z_i$ by using the $QR$ decomposition of the matrix $\tilde{H_i}.$ From the equation (\ref{eq1b}) note that a vector $z_i$ satisfies the following  normal system of
equations:
\begin{equation}\label{eq2}\relax
V_i^\ast A^\ast AV_iz_i = V_i^\ast A^\ast b = \beta V_i^\ast A^\ast
V_ie_1.
\end{equation}
Since the columns of $V_i$ are orthonormal, and $v_{i+1}$ is orthogonal to the columns of $V_i,$ by using the equation (\ref{eq1a}), the above equation 
can be rewritten as follows in the terms of $H_i:$ 
\begin{equation}\label{eq3}\relax
(H_i^\ast H_i+|h_{i+1,i}|^2 e_ie_i^\ast)z_i = \beta H_i^\ast e_1.
\end{equation}
From the equation ~(\ref{eq1}) note that the norm of a residual vector associated with $V_iz_i$ is smaller over the Krylov subspace of dimension $i.$ 
\section{The new Least Squares problem}
In this section, we devise a least squares problem that connects approximate solutions at two successive iterations of GMRES.
Throughout this section, iteration number is fixed at $m,$ and
$`y`$ denotes the solution of a least squares problem in the 
equation~(\ref{eq2}), for $i=m.$ 

Consider the following least squares problem:
$$z = \argmin_{x \in \textcal{C}^m} \|b-AV_m(I-e_me_m^\ast)x\|.$$
Note that a solution vector $z$ of this least squares problem satisfies
the following system of normal equations:
\begin{equation}\label{eq5}\relax
(I-e_m e_m^\ast)V_m^\ast A^\ast (b-AV_m(I-e_me_m^\ast)z)=0.
\end{equation}
As linear span of a vector $e_m$ is the null space of a projection operator $(I-e_me_m^\ast),$ it gives the following:
$$V_m^\ast A^\ast (b-AV_m(I-e_me_m^\ast)z) = Ke_m,$$
where $K$ is a scalar that can be obtained by applying an inner product with a vector $e_m$ on both the sides of the above equation. The next theorem makes a connection between solutions of a new least squares problem and the usual least squares problem of GMRES for the vector space spanned by columns of $V_{m-1}:= \begin{bmatrix} v_1~ v_2~ \cdots ~v_{m-1}
\end{bmatrix}.$ 
\begin{thm}\label{thm1}\relax
Let a vector $z$ be the same as in the equation (\ref{eq5}), and $z_{m-1}$ be a vector of length $m-1$ 
such that $V_m (I-e_me_m^\ast)z = V_{m-1}z_{m-1}.$ Then 
\begin{equation}\label{V1}\relax
z_{m-1}= \arg\min_{x \in \mathcal{C}^{m-1}}\|b-AV_{m-1}x\|^2
\end{equation}
\end{thm}
\begin{proof}
From the equation (\ref{eq5}) we have 
$$(I-e_me_m^\ast)V_m^\ast A^\ast AV_m(I-e_me_m^\ast)z = (I-e_me_m^\ast)V_m^\ast A^\ast b. $$ 
On substituting $V_m (I-e_me_m^\ast)z = V_{m-1}z_{m-1}$
the previous equation gives
\begin{equation}\label{N2}\relax
(I-e_me_m^\ast)V_m^\ast A^\ast A V_{m-1}z_{m-1} = (I-e_me_m^\ast)V_m^\ast
A^\ast b.
\end{equation}
As linear span of a vector $e_m$ is the null space of a projection operator $(I-e_me_m^\ast),$ this equation gives the following for some scalar $K:$ 
\begin{equation}\label{N1}\relax
V_m^\ast A^\ast (b-AV_{m-1}z_{m-1})= Ke_m. 
\end{equation}
Thus the vector $V_m^\ast A^\ast (b-AV_{m-1}z_{m-1})$ is parallel to the vector $e_m.$ Since $e_i^\ast e_m=0$ for $i=1,2, \ldots, m-1,$ the above  equation gives
$$V_{m-1}^\ast A^\ast A(b-AV_{m-1}z_{m-1}) =0.$$ Therefore,
the vector $z_{m-1}$ is a solution of the least squares problem in
(\ref{V1}). Hence, the theorem proved.
\end{proof}
Observe from the equation~(\ref{N1}) that $K=v_m^\ast A^\ast r_{m-1} =~<Av_m, r_{m-1}>, $ where $r_{m-1}=b-AV_{m-1}z_{m-1}.$
The next theorem relates an approximate solution at the $m^{th}$ iteration of GMRES to residual norms in the $(m-1)^{th}$ and $m^{th}$ iterations. 
\begin{thm}\label{thm2}\relax
Let $r_{m-1}$ and $r_m$ be residual vectors at $(m-1)^{th}$ and $m^{th}$
iterations of GMRES, respectively. Assume that column vectors of the matrix $AV_m$ are linearly independent. If a vector $y$ is the solution of the least
squares problem at $m^{th}$ iteration of GMRES then 
$$\|r_{m-1}\|^2-\|r_m\|^2 = Ke_m^\ast y.$$
\end{thm}
\begin{proof}
From the hypothesis of the theorem we have $V_m^\ast A^\ast AV_my = V_m^\ast A^\ast b.$ Observe that by using this, the equation~(\ref{N1}) gives
$$V_m^\ast A^\ast AV_m\Big(y-\begin{pmatrix}
z_{m-1}\\0
\end{pmatrix}\Big)=Ke_m. $$
As column vectors of the matrix $AV_m$ are linearly independent the above equation gives the following:
\begin{equation}\label{avmpos}\relax
y-\begin{pmatrix}
z_{m-1}\\0
\end{pmatrix}= K (V_m^\ast A^\ast AV_m)^{-1} e_m.
\end{equation}
This implies 
$$AV_m \Big(y-\begin{pmatrix}
z_{m-1}\\0
\end{pmatrix}\Big)= K.AV_m (V_m^\ast A^\ast AV_m)^{-1} e_m,$$
and 
\begin{equation}\label{eq7}\relax
r_{m-1}-r_m= K. AV_m (V_m^\ast A^\ast AV_m)^{-1} e_m.
\end{equation}
In the above equation we used the relations $r_m= b-AV_my,$ and the following:
\begin{equation*}\label{resm1}\relax
r_{m-1}=b-AV_{m-1}z_{m-1}= b-AV_m\begin{pmatrix}
z_{m-1}\\0
\end{pmatrix}.
\end{equation*}
Now, apply an inner product with
$r_0=b$ on both the sides of the equation (\ref{eq7}). It gives \begin{equation}\label{eqst}\relax
r_{m-1}^\ast r_0-r_m^\ast r_0= Ke_m^\ast (V_m^\ast A^\ast AV_m)^{-1}
V_m^\ast A^\ast b.
\end{equation}
By using $V_m^\ast A^\ast AV_my = V_m^\ast A^\ast b,$ observe that the right-hand side expression in the above equation is $Ke_m^\ast
y.$ The proof will be complete if
$$r_{m-1}^\ast r_0-r_m^\ast r_0 = \|r_{m-1}\|^2-\|r_m\|^2.$$
From the Theorem-\ref{thm1} we know that $r_{m-1}=b-AV_{m-1}z_{m-1}$ is orthogonal to $AV_{m-1}z_{m-1}=r_{m-1}-b=r_{m-1}-r_0.$ Similarly, the residual vector $r_m$ is orthogonal to $r_m-r_0.$ Thus, 
$r_i^\ast r_0 = \|r_i\|^2~~\mbox{for}~i=m-1,~m.$ Therefore, the above equation holds true, and the proof is over.
\end{proof}
\section{Equality of Residual and Harmonic Residual vectors}
In this section, we define Harmonic Residual vectors and will discuss the  coincidence of these vectors with a residual vector in GMRES. As in the previous section, we fix iteration number in GMRES as $m$ and will use the notation of the Section-2.
\begin{defn}\label{defn1}\relax
The $m$ eigenvalues $\{\sigma_j\}_{j=1}^m$ of the generalized eigenvalue problem 
$$V_m^\ast A^\ast AV_mu = \sigma V_m^\ast A^\ast V_m u $$
are called the Harmonic Ritz values at iteration $m$ of GMRES. The vectors $\{u_j\}_{j=1}^m$ are called the Harmonic Ritz vectors. The pair $(\sigma,u)$ is called the Harmonic Ritz pair.
\end{defn}
\begin{defn}
Let $(\sigma,u)$ be a Harmonic Ritz pair at iteration $m$ of GMRES. The vector $AV_mu-\sigma V_mu$ is called the Harmonic residual vector at iteration $m$ of GMRES.
\end{defn}
The following theorem derives a necessary condition for the equality of a Harmonic resdiual vector $AV_mu-\sigma V_mu$ and $b-AV_my,$ the residual vector at $m^{th}$ iteration of GMRES.
\begin{thm}\label{nece}\relax
Let a Harmonic residual vector $AV_mu-\sigma V_mu$ be the same as the residual vector $b-AV_my$ at $m^{th}$ iteration of GMRES. Then $e_m^\ast y =-e_m^\ast u.$
\end{thm}
\begin{proof}
Consider $AV_mu-\sigma V_mu=b-AV_my.$ By using $b= \beta V_me_1$ and the equation (\ref{eq1a}) for $i=m,$ this implies
$$V_mH_mu+h_{m+1,m}v_{m+1}e_m^\ast u -\sigma V_m u = \beta V_m e_1 -V_mH_my-h_{m+1,m}v_{m+1}e_m^\ast y.$$
Recall that columns of the matrix $V_m$ form an orthonormal basis for the  Krylov subspace $\textcal{K}_m(A,b).$ As $v_{m+1}$ is orthogonal to column vectors of $V_m,$ the above equation implies the following:
$$H_mu-\sigma u = \beta e_1-H_my~\mbox{and}~e_m^\ast u = -e_m^\ast y.$$ Hence, the theorem proved.
\end{proof}
In what follows, we prove that $e_m^\ast y =-e_m^\ast u$ is a sufficient condition for the equality of the vectors in the Theorem-\ref{nece}.
\begin{thm}\label{suffnostag}\relax
Let $AV_mu-\sigma V_mu$ be a Harmonic residual vector, and $b-AV_my$ be the residual vector at $m^{th}$ iteration of GMRES. Assume that there is no stagnation at $m^{th}$ iteration, and $e_m^\ast y =-e_m^\ast u.$ Then  $AV_mu-\sigma V_mu=b-AV_my.$
\end{thm}
\begin{proof}
As $(\sigma,u)$ is a Harmonic Ritz pair, by using the Definition \ref{defn1} and the equation (\ref{eq1a} for $i=m$, it satisfy the following equation:
\begin{equation}\label{harh}\relax
H_m^\ast H_m u+|h_{m+1,m}|^2e_me_m^\ast u = \sigma H_m^\ast u.
\end{equation}
Similarly, as $b-AV_my$ is the residual vector at $m^{th}$ iteration, by using the equation (\ref{eq3}) for $i=m,$ the vector $y$ satisfies the following equation:
\begin{equation}\label{res}\relax
(H_m^\ast H_m+|h_{m+1,m}|^2 e_me_m^\ast)y = \beta H_m^\ast e_1.
\end{equation}
By using $e_m^\ast y =-e_m^\ast u,$ the above two equations imply the following relation:
$$H_m^\ast(H_mu-\sigma u+H_m y-\beta e_1) = 0. $$
As there is no stagnation at $m^{th}$ iteration, by using the Lemma-\ref{thm3} and the Theorem-\ref{thm4} from the Section 4, observe that $H_m^\ast $ is a non-singular matrix. Thus, the above equation implies
$H_mu-\sigma u=\beta e_1-H_m y. $ On multiplying both the sides with a matrix $V_m$ gives $V_mH_mu-\sigma V_m u=\beta V_m e_1-V_mH_m y.$
Now, by using the equation  (\ref{eq3}) for $i=m,$ this equation can be rewritten as follows:
$$AV_mu-h_{m+1,m}v_{m+1}e_m^\ast u -\sigma V_m u=\beta V_m e_1-AV_my+h_{m+1,m}v_{m+1}e_m^\ast y. $$
As $ \beta V_m e_1 = b$ and $e_m^\ast y =-e_m^\ast u,$ this gives $AV_mu-\sigma V_mu=b-AV_my,$ the required equation. Hence, the proof is over.
\end{proof}
Observe that the Theorems-\ref{nece} and \ref{suffnostag} can be generalized  to the following:
\begin{thm}\label{gene}\relax
Let $AV_mu-\sigma V_mu$ be a Harmonic residual vector, and $b-AV_my$ be the residual vector at $m^{th}$ iteration of GMRES. Assume that there is no stagnation at $m^{th}$ iteration. Then for some non-zero scalar $K,$ $AV_mu-\sigma V_mu=K(b-AV_my),$ if and only if $e_m^\ast u =-K(e_m^\ast y).$
\end{thm}
The proof of theorem-\ref{gene} is trivial from the proofs of the Theorems-\ref{nece} and \ref{suffnostag}. The following example illustrates the Theorem-\ref{gene}.
\begin{eg}
Consider the system of equations $Ax=b,$ where
\begin{equation*}
A = \begin{bmatrix}
1~ & ~1~ & ~1\\
1  & 0 & 1 \\
0 & 1 & 1 
\end{bmatrix},
\end{equation*}
and $b=[1~ 0~ 0]'=e_1.$ Similarly, we use $e_2$ and $e_3$ to represent the vectors $[0~1~0]'$ and $[0~0~1]'$ respectively. Let the zero vector is an approximate solution so that $e_1$ is the initial residual vector $r_0$ in GMRES.

Starting with $e_1$, for the matrix $A$ the Arnoldi algorithm gives the following after 2 iterations: 
\begin{equation*}
AV_2 = V_2H_2 + e_3[0~~1]
\end{equation*}
where $V_2 = [e_1,e_2]$, and
\begin{equation*}H_2 =
\begin{bmatrix}
1~ & ~1 \\
1 &  0 
\end{bmatrix}.
\end{equation*}
The residual vector at the $2^{nd}$ iteration of GMRES is $$r_2=b-AV_2[1/3~1/3]'=[1/3~-1/3~~-1/3]'.$$ This example do not have the stagnation of GMRES. The harmonic Ritz values at $2^{rd}$ iteration of GMRES are $\pm \sqrt{3},$ and $ V_2[\frac{\sqrt{3}+1}{2}~~1/2]'=[\frac{\sqrt{3}+1}{2}~~1/2~~0]'$ is the harmonic Ritz vector corresponding to $\sqrt{3}$. Thus, the corresponding harmonic residual vector is 
$$A \begin{bmatrix}
\frac{\sqrt{3}+1}{2}\\1/2\\0
\end{bmatrix}
-\sqrt{3}\begin{bmatrix}
\frac{\sqrt{3}+1}{2}\\1/2\\0
\end{bmatrix}=\begin{bmatrix}
-1/2\\1/2\\1/2
\end{bmatrix}.$$
Observe that for this example  the above vector is equal to $\frac{-3}{2}* r_2,$ and $e_2^\ast u=1/2,$ $e_2^\ast y=1/3.$ Thererfore, the vectors $u$ and $y$ satisfy the necessary and sufficient condition in the Theorem-\ref{gene}.
\end{eg}

The Theorem-\ref{suffnostag} has given a sufficient condition in the absence of the stagnation for the coincidence of a Harmonic residual vector and the residual vector of GMRES at the $m^{th}$ iteration. The following theorem  discusses this coincidence of vectors in the presence of the stagnation. We delay its proof to the next section for the convenience.
\begin{thm}\label{suffstag}\relax
Let $AV_mu-\sigma V_mu$ be a Harmonic residual vector, and $b-AV_my$ be the residual vector at $m^{th}$ iteration of GMRES. Assume that there is a stagnation at $m^{th}$ iteration. Then  $AV_mu-\sigma V_mu=b-AV_my+\xi V_ms_2,$ where $s_2$ is a vector such that $V_m^\ast A^\ast V_ms_2=0,$ and $\xi$ is some scalar.
\end{thm}
\section{The stagnation of GMRES}
This section derives a necessary and sufficient condition on harmonic Ritz vectors for the stagnation of GMRES. For this, it first derives a necessary and sufficient condition on $H_m,$ when residuals at $(m-1)^{th}$ and $m^{th}$ iterations of GMRES are stagnated, that means $\|r_{m-1}\|= \|r_m\|.$  From the Theorem-\ref{thm2}, observe  that this happens if and only if $K=0$ or $e_m^\ast y =0.$  The next theorem shows that if either $K$ or $e_m^\ast y$ is zero, then the other one also equal to the zero.
\begin{lem}\label{thm3}\relax
Let a scalar $K$ and a vector $y$ be the same as in the Theorem-\ref{thm2}. Then, $K=0$ if and only if $e_m^ \ast y =0.$
\end{lem}
\begin{proof}
First we prove $e_m^\ast y =0,$ if $K=0.$ By using the equation~(\ref{avmpos}), note that $K=0$ implies $$y=\begin{pmatrix}
z_{m-1}\\0
\end{pmatrix}.$$ Recall that $z_{m-1}$ is a vector of length $(m-1).$ Therefore, $e_m^ \ast y =0.$ Next, we prove the converse, that means, $K=0$ if $e_m^\ast y =0.$  
By using the equation (\ref{avmpos}), observe that $e_m^ \ast y =0$ implies $$K e_m^\ast (V_m^\ast A^\ast AV_m)^{-1} e_m =0.$$
As column vectors of the matrix $AV_m$ are linearly independent, the matrix $ (V_m^\ast A^\ast AV_m)^{-1}$ is a positive definite matrix. Therefore, this implies $K=0.$ Hence, the proof is over.
\end{proof}
By using the Lemma-\ref{thm3}, the following theorem proves that the stagnation at $m^{th}$ iteration of GMRES occurs if and only if $H_m$ is a singular matrix.
\begin{thm}\label{thm4}\relax
Let the vectors $r_{m-1},r_m$ and $y$ be the same as in the Theorem-\ref{thm2}, and $H_m$ is an upper Hessenberg matrix at $m^{th}$ iteration of GMRES. Assume that $r_{m-1} \neq 0.$ Then $e_m^\ast y=0$ if and only if $H_m$ is a singular matrix.
\end{thm}
\begin{proof}
Let $e_m^\ast y=0.$ By using the equation~(\ref{eq3}) for $i=m,$ this gives the following equation:
\begin{equation}\label{eq8b}\relax
H_m^\ast H_my = \beta H_m^\ast e_1.
\end{equation}
If $H_m$ is non-singular, this equation implies $H_my =
\beta e_1,$ and $y= \beta H_m^{-1}e_1.$ This together with the equation (\ref{eq1a}) gives 
$$r_m = b-AV_my=b-V_mH_my-h_{m+1,m}v_{m+1}e_m^\ast y = b-\beta V_me_1-h_{m+1,m}v_{m+1}e_m^\ast y.$$
By using the fact that $b=\beta V_me_1,$ and $e_m^\ast y=0,$ this gives $r_m=0.$ Further, 
the Theorem-\ref{thm2} implies $r_{m-1}=0,$ a contradiction to the hypothesis of the theorem that $r_{m-1} \neq 0.$ Therefore, $H_m$ is a singular matrix.

\ni Now, we prove the converse. Let $H_m$ be a singular matrix. Then there exists a non-zero vector $s$ such that $H_ms=0.$  As $H_m$ is
an unreduced upper Hessenberg matrix this implies
$$e_m^\ast s \neq 0.$$  Otherwise, $H_ms=0$ implies $s=0.$ Now, take an inner product with $s$ on both the sides of the equation~(\ref{eq1a}) for $i=m.$  This
gives
$$s^\ast H_m^\ast H_m y +|h_{m+1,m}|^2 s^\ast e_me_m^\ast y = \beta s^\ast H_m^\ast e_1.$$
By using $H_ms=0$ and $e_m^\ast s \neq 0,$ this equation implies $e_m^\ast y =0.$
Therefore, the
theorem proved. 
\end{proof}
The following theorem derives a necessary and sufficient condition on harmonic Ritz vectors for the stagnation at $m^{th}$ iteration of GMRES. For this, it uses the Lemma-\ref{thm3} and the Theorem-\ref{thm4}.
\begin{thm}\label{yu}\relax
Let $(\sigma, u)$ be a harmonic Ritz pair at $m^{th}$ iteration of GMRES. Assume that $b-AV_my$ is the residual vector at the same iteration.
If $e_m^\ast  y=0$ then $e_m^\ast u=0.$
\end{thm}
\begin{proof}
Let $e_m^\ast  y=0.$ The proof for $e_m^\ast u=0 $ is required. As $e_m^\ast y=0$ from the Lemma-\ref{thm3} and the Theorem-\ref{thm4} note that $H_m$ is a singular matrix. Assume that $H_m$ is of the following form: 
\begin{equation}\label{hm}\relax
H_m := \begin{bmatrix}
H_{m-1}&h\\
\gamma e_{m-1}^\ast &\alpha \gamma
\end{bmatrix},
\end{equation}
where $H_{m-1}$ is a  principal submatrix of order $m-1$ from the top left corner of $H_m.$ The singularity of a matrix $H_m$ implies the existence of a vector $s_1$ such that 
$$h= H_{m-1}s_1.$$ As $(\sigma,u)$ is a harmonic Ritz pair, it satisfies the equation (\ref{harh}). By using the form of a matrix $H_m$ in the above equation,  (\ref{harh}) can be written as follows:
\begin{eqnarray}\label{hsh}\relax
\begin{bmatrix}
H_{m-1}^\ast H_{m-1}+|\gamma|^2e_{m-1}e_{m-1}^\ast & H_{m-1}^\ast h+\alpha|\gamma|^2e_{m-1}\\
h^\ast H_{m-1}+\bar{\alpha}|\gamma|^2e_{m-1}^\ast&h^\ast h+|\alpha\gamma|^2
\end{bmatrix}u = \nonumber \\ \sigma \begin{bmatrix}
H_{m-1}^\ast&\bar{\gamma} e_{m-1}\\h^\ast&\bar{\alpha}\bar{\gamma}
\end{bmatrix} u-|h_{m+1,m}|^2e_me_m^\ast u.
\end{eqnarray}
Assume that $u_{1:m-1}$ represents a vector whose entries are the same as first $m-1$ elements of a vector $u,$ and $u_m$ denotes a last entry of the vector $u.$ Following this notation, the comparison of both the sides of the above equation gives the following relations:
\begin{eqnarray}\label{m1h}\relax
H_{m-1}^\ast H_{m-1}u_{1:m-1}-\sigma H_{m-1}^\ast u_{1:m-1}+|\gamma|^2e_{m-1}(e_{m-1}^\ast u_{1:m-1}) \nonumber \\=u_m(\sigma \bar{\gamma} e_{m-1}- H_{m-1}^\ast h-\alpha|\gamma|^2e_{m-1}),
\end{eqnarray}
and
\begin{eqnarray*}
h^\ast H_{m-1}u_{1:m-1}+\bar{\alpha}|\gamma|^2e_{m-1}^\ast u_{1:m-1}+(h^\ast h+|\alpha\gamma|^2)u_m= \nonumber \\ \sigma h^\ast u_{1:m-1}+\sigma \bar{\alpha}\bar{\gamma}u_m-|h_{m+1,m}|^2 u_m .
\end{eqnarray*}
On substituting $h= H_{m-1}s_1,$ this implies
$$s_1^\ast (H_{m-1}^\ast H_{m-1}-\sigma H_{m-1}^\ast)u_{1:m-1}+\bar{\alpha}|\gamma|^2e_{m-1}^\ast u_{1:m-1}+(h^\ast h+|\alpha\gamma|^2-\sigma\bar{\alpha}\bar{\gamma}+|h_{m+1,m}|^2)u_m=0. $$
As $H_m$ is an unreduced upper Hessenberg singular matrix, from the equations (\ref{hm}) and $h= H_{m-1}s_1,$ note that $\alpha$ and $\gamma$ are non-zero, and $\bar{\alpha} = s_1^\ast e_{m-1}.$ Further, apply an inner product on both the sides of the equation (\ref{m1h}) with a vector $s_1.$ Then, substituting it in the above equation gives
\begin{equation}\label{nzero}\relax
u_m(\sigma \bar{\gamma} s_1^\ast e_{m-1}-h^\ast h-|\alpha\gamma|^2+h^\ast h+|\alpha\gamma|^2-\sigma\bar{\alpha}\bar{\gamma}+|h_{m+1,m}|^2)=0.
\end{equation}
The above equation has used the relation $h= H_{m-1}s_1$ to obtain the second term inside the parentheses. Using $\bar{\alpha} = s_1^\ast e_{m-1}$ and $h_{m+1,m} \neq 0$ observe that the term inside the parentheses of the  equation (\ref{nzero}) is non-zero. Therefore, $u_m := e_m^\ast u = 0.$ Hence,  the theorem proved.
\end{proof}
Observe that in the Theorem- \ref{yu} $\sigma \neq 0$ is not necessary as the equation (\ref{nzero}) holds true for $\sigma = 0 $ as well. Next, the following theorem proves the converse of the Theorem-\ref{yu}.
\begin{thm}\label{uy}\relax
Let vectors $u$ and $y$ be the same as in the Theorem-\ref{yu}. If $e_m^\ast  u=0$ then $e_m^\ast y=0.$
\end{thm}
\begin{proof}
As $(\sigma,u)$ is a harmonic Ritz pair, it satisfies the equation (\ref{harh}). Further, using $e_m^\ast  u=0$ it gives $H_m^\ast H_m u = \sigma H_m^\ast u.$ This implies either $H_mu-\sigma u$ is a zero vector or $H_m$ is a singular matrix. Assume that $H_mu-\sigma u$ is a zero vector. Then, as $H_m$ is an unreduced upper Hessenberg matrix and $e_m^\ast u=0,$ by using the Lemma-2.1 in \cite{rbl}, the equation $H_mu=\sigma u$ implies $u$ is a zero vector, a contradiction to the statement that $u$ is a harmonic Ritz vector. Therefore, $H_m $ is a singular matrix. Now, by using the Theorem-\ref{thm4}, this gives $e_m^\ast y=0.$ Hence, the proof is over.
\end{proof}
The Theorems-\ref{thm3}, \ref{yu}, and \ref{uy} have shown that the stagnation occurs at $m^{th}$ iteration of GMRES if and only if $e_m^\ast u=0$ and $e_m^\ast y=0.$ That means, when the stagnation occurs, the necessary and sufficient condition in the Theorems-\ref{nece} and \ref{suffnostag} for the coincidence of a harmonic residual vector and the residual vector in GMRES is trivial. The following is the proof for the Theorem-\ref{suffstag} of the previous section.
\begin{proof}[\textbf{Proof of Theorem-\ref{suffstag}}]
As $(\sigma,u)$ is a harmonic Ritz pair, and $b-AV_my$ is a residual at $m^{th}$ iteration, the vectors $u$ and $y$ satisfy the equations (\ref{harh}) and (\ref{res}) respectively. Since there is a stagnation at $m^{th}$ iteration of GMRES, the Theorems-\ref{thm3}, \ref{yu}, and \ref{uy} imply $e_m^\ast u=0$ and $e_m^\ast y=0.$ Thus, $H_m^\ast H_mu= \sigma H_m^\ast u,$ and $H_m^\ast H_my = \beta H_m^\ast e_1.$  These two equations together imply $H_m^\ast(H_mu- \sigma H_m^\ast u+H_my - \beta H_m^\ast e_1)=0.$ This implies $$H_mu- \sigma u= \beta e_1-H_my +\xi s_2.$$ Here, $\xi$ is a scalar, and $s_2$ is a vector such that $H_m^\ast s_2=0.$ Note that a vector $s_2$ exists due to the Theorem-\ref{thm4}, and the stagnation of GMRES. On multiplying both the sides of the above equation with a matrix $V_m$ gives $V_mH_mu-\sigma V_m u=\beta V_m e_1-V_mH_m y+\xi V_m s_2.$
Now, by using $e_m^\ast u=0,$ $e_m^\ast y=0,$ and the equation  (\ref{eq3}) for $i=m,$ this equation can be written as follows:
$$AV_mu-\sigma V_m u=\beta V_m e_1-AV_my+\xi V_m s_2 $$
As $ \beta V_m e_1 = b,$ this gives $AV_mu-\sigma V_mu=b-AV_my+\xi V_m s_2,$ the required equation. Hence, the proof is over.
\end{proof}
In the following, we prove the theorems those relate harmonic Ritz vectors at any two successive iterations of GMRES in the presence of the stagnation.
\begin{lem}\label{har2success}\relax
Assume that the stagnation has occurred at the $m^{th}$ iteration of GMRES. Let $u$ be a harmonic Ritz vector corresponding to the non-zero harmonic Ritz value $\sigma,$  at $m^{th}$ iteration. Then $(\sigma,u_{1:m-1})$ is a harmonic Ritz pair at $(m-1)^{th}$ iteration of GMRES. 
\end{lem}
\begin{proof}
Due to the stagnation at $m^{th}$ iteration of GMRES the Lemma-\ref{thm3} and the Theorem-\ref{yu} give $u_m:=e_m^\ast u=0.$ Substituting this in the equation (\ref{m1h}) gives the desired result, that means $(\sigma,u_{1:m-1})$ is a harmonic Ritz pair at $(m-1)^{th}$ iteration.
\end{proof}
Next, in the following, we prove the converse of the Lemma-\ref{har2success}.
\begin{lem}
Let $u$ be a harmonic Ritz vector corresponding to the non-zero harmonic Ritz value $\sigma,$  at $m^{th}$ iteration. If $(\sigma,u_{1:m-1})$ is a harmonic Ritz pair at $(m-1)^{th}$ iteration of GMRES then there is a stagnation at $m^{th}$ iteration of GMRES.
\end{lem}
\begin{proof}
From the hypothesis of the lemma and the equation (\ref{m1h}) we have $e_m^\ast u=0.$ By using the Theorem-\ref{uy}, this gives $e_m^\ast y=0.$ Now, use the Theorem-\ref{thm2} to conclude $r_{m-1}=r_m,$ where $r_i$ is a residual at the $i^{th}$ iteration. Therefore, there is a stagnation at $m^{th}$ iteration. Hence, the proof is over.
\end{proof}
\section{Conclusions}
This paper shows that coincidence of the GMRES residual vector and Harmonic residual vector is theoretically possible, and derives the necessary and sufficient condition for this coincidence. Then, for the stagnation in GMRES, it derives necessary and sufficient conditions those based on elements of a harmonic Ritz vector. Further, it shows that in case of the stagnation, the harmonic Ritz vectors corresponding to non-zero harmonic Ritz values are preserved. The procedure followed in this paper for proving these results will be helpful for the study of the near stagnation of GMRES in terms of elements of harmonic Ritz vectors.
\section*{Acknowledgements}
The author thanks the National Board of Higher Mathematics, India
for supporting this work under the Grant number 02/40(3)/2016. 
\section*{References}
%

\end{document}